\documentclass{amsart}

\usepackage{amssymb,stmaryrd,mathrsfs,dsfont,amsmath}

\theoremstyle{plain}
\newtheorem{theorem}{Theorem}[section]
\newtheorem{lemma}[theorem]{Lemma}
\newtheorem{proposition}[theorem]{Proposition}
\newtheorem{corollary}[theorem]{Corollary}

\theoremstyle{definition}

\newtheorem{definition}[theorem]{Definition}

\theoremstyle{remark}
\newtheorem{remark}[theorem]{Remark}

\input xy
\xyoption{all}

\begin{document}

\title[Semigroups of Transformations that preserve a partition]
{On certain Semigroups of Transformations that preserve a partition}

\author[Mosarof Sarkar]{\bfseries Mosarof Sarkar}
\address{Department of Mathematics, Central University of South Bihar, Gaya,Bihar, India}
\email{mosarofsarkar@cusb.ac.in}

\author[Shubh N. Singh]{\bfseries Shubh N. Singh}
\address{Department of Mathematics, Central University of South Bihar, Gaya, Bihar, India}
\email{shubh@cub.ac.in}


\begin{abstract}

Let $X$ be a nonempty set, and let $\mathcal{T}_X$ be the full transformation semigroup on $X$. For a partition $\mathcal{P} = \{X_i \;|\; i\in I\}$ of $X$, we consider the semigroup $T(X, \mathcal{P}) = \{f\in \mathcal{T}_X\;|\; \forall X_i\;\exists X_j,\; X_i f \subseteq X_j\}$, the subsemigroup $\Sigma(X, \mathcal{P}) = \{f\in T(X, \mathcal{P})\;|\; Xf \cap X_i \neq \emptyset\; \forall X_i\}$, and the group of units $S(X, \mathcal{P})$ of $T(X, \mathcal{P})$. In this paper, we first characterize the elements of $\Sigma(X, \mathcal{P})$. For a permutation $f$ of finite $X$, we next observe whether there exists a nontrivial partition $\mathcal{P}$ of $X$ such that $f\in S(X, \mathcal{P})$. We then characterize and enumerate the idempotents in the semigroup $\Sigma(X, \mathcal{P})$ for arbitrary and finite $X$, respectively. We also characterize the elements of $S(X, \mathcal{P})$. For finite $X$, we finally calculate the cardinality of $T(X, \mathcal{P})$, $\Sigma(X, \mathcal{P})$, and $S(X, \mathcal{P})$.
\end{abstract}



\maketitle

\section{Introduction}
We assume that the reader is familiar with elementary concepts of combinatorics and semigroup theory. Throughout the paper, let $X$ denote a set with more than two elements, and let $\mathcal{P}$ denote a partition of $X$. The symbols $\mathcal{T}_X$ and $\mathcal{S}_X$ will be used to denote the full transformation semigroup and the symmetric group on $X$, respectively. For a subset $A\subseteq X$, we denote by $Af$ the image of $A$ under $f \in \mathcal{T}_X$. A map $f\in \mathcal{T}_X$ \emph{preserves} a partition $\mathcal{P}$ if for every $X_i \in \mathcal{P}$, there exists $X_j \in \mathcal{P}$ such that $X_i f \subseteq X_j$.

\vspace{0.2cm}
The notion of semigroups of transformations that preserve a partition is one of the most trending topics of semigroup theory. For a partition $\mathcal{P}$ of a set $X$, Pei \cite{pei-s94} introduced and studied the subsemigroup
\[T(X, \mathcal{P}) = \{f\in \mathcal{T}_X\;|\; \forall X_i \in \mathcal{P}, \; \exists X_j \in \mathcal{P}, \;\; X_i f \subseteq X_j\}\]
of $\mathcal{T}_X$. Moreover, Pei proved in \cite[Theorem 2.8]{pei-s94} that $T(X, \mathcal{P})$ is exactly the semigroup of all continuous selfmaps on $X$ endowed with the topology having $\mathcal{P}$ as a basis. Since then, $T(X, \mathcal{P})$ and its subsemigroups have received considerable attention and their several fascinating algebraic and combinatorial aspects have been investigated (see for example \cite{arjo15, arjo-c04, arjo09, east-c16, east16, fernand11, fernand14, pei-s98, pei-s05, pei-c05, sun-s07, sun-jaa13}). There have also been a number of interesting works on certain generalizations of the semigroup $T(X, \mathcal{P})$ (see for example \cite{deng11, deng12, peidin-s05, pei-c07, puri16}).

\vspace{0.2cm}

Pei \cite{pei-c05} studied the regularity and Green's relations in the semigroup $T(X, \mathcal{P})$. When $\mathcal{P}$ is a uniform partition of finite $X$, Pei \cite{pei-s05} gave an upper bound for the rank of $T(X, \mathcal{P})$. Later, Ara\'{u}jo et al. \cite{arjo09} calculated the rank of $T(X, \mathcal{P})$ and thus settled a conjecture on the rank of $T(X, \mathcal{P})$ posed by Pei in \cite{pei-s05}. Ara\'{u}jo et al. \cite{arjo15} also calculated the rank of $T(X, \mathcal{P})$ for an arbitrary partition $\mathcal{P}$ of finite $X$. Dolinka et al. characterized as well as enumerated the idempotents of the semigroup $T(X, \mathcal{P})$ for finite set $X$ in \cite{east-c16} and \cite{east16} for the uniform and non-uniform cases, respectively. The cardinality of particular classes of subsemigroups of the semigroup $T(X, \mathcal{P})$ have also been calculated (see for example \cite{fernand12, sun-13}).

\vspace{0.2cm}

Pei \cite{pei-s05} also considered the group of units $S(X, \mathcal{P})$ of the semigroup $T(X, \mathcal{P})$ and observed that $S(X, \mathcal{P})$ is exactly the subgroup of all homeomorphisms on $X$ endowed with the topology having $\mathcal{P}$ as a basis, and called it the \emph{homeomorphism group}. For a uniform partition $\mathcal{P}$ of finite $X$, Pei \cite{pei-s05} further deduced an upper bound for the rank of $S(X, \mathcal{P})$. Later, Ara\'{u}jo et al. \cite{arjo09} calculated the rank of $S(X, \mathcal{P})$ when $\mathcal{P}$ is a uniform partition of finite $X$. The homeomorphism group $S(X, \mathcal{P})$ is also studied by Ara\'{u}jo et al. in \cite{arjo15}.

\vspace{0.2cm}

For a partition $\mathcal{P} = \{X_i\;|\; i\in I\}$ of a set $X$, let
\[\Sigma(X, \mathcal{P}) = \{f\in T(X, \mathcal{P})\;|\; Xf \cap X_i \neq \emptyset\;\; \forall X_i \in \mathcal{P}\}.\]
It is clear that $\Sigma(X, \mathcal{P})$ is a subsemigroup of $T(X, \mathcal{P})$. When $\mathcal{P}$ is a uniform partition of finite $X$, Ara\'{u}jo et al. \cite{arjo09} calculated the rank of the semigroup $\Sigma(X, \mathcal{P})$. For a finite $X$, some interesting properties of $\Sigma(X, \mathcal{P})$ are also investigated in \cite{arjo15}.

\vspace{0.2cm}

The rest of the paper is organized as follows. In Section $2$, we recall necessary concepts from semigroup theory and combinatorics and introduce notation used within the paper. In Section $3$, we characterize the elements of the semigroup $\Sigma(X, \mathcal{P})$. For a permutation $f$ of finite $X$, we observe whether there exists a nontrivial partition $\mathcal{P}$ of $X$ such that $f\in S(X, \mathcal{P})$ in Section $4$. In Section $5$, we characterize and enumerate the idempotents in the semigroup $\Sigma(X, \mathcal{P})$ for arbitrary and finite $X$, respectively. Moreover, we characterize the elements of $S(X, \mathcal{P})$.
In Section $6$, we finally calculate the cardinality of $T(X, \mathcal{P})$, $\Sigma(X, \mathcal{P})$, and $S(X, \mathcal{P})$ when $X$ is a finite set.

\vspace{0.2cm}
\section{Preliminaries and Notation}

In this section, we introduce relevant notation and recall concepts from combinatorics and semigroup theory that are requisite to the paper. We refer the reader to the standard books \cite{bru10, howie95} for more detailed information from combinatorics and semigroup theory, respectively.

\vspace{0.2cm}

Unless stated otherwise, we will use capital letter to denote nonempty subset, calligraphic letter to denote collection of subsets, and small letter to denote set-element, map, or positive integer. The letter $I$ will be reserved for an arbitrary indexing set. The set of all positive integers is denoted by $\mathbb{N}$. We will always presume that $m$ and $n$ are positive integers. The symbol $I_m$ denote the subset $\{1,\ldots, m\}$. The number of elements of a set $A$ is denoted by $|A|$ and is called the \emph{cardinality} or \emph{size} of $A$. A set of cardinality $n$ is called an \emph{$n$-element set}. We write $A\setminus B$ to denote the set of all elements $x \in A$ such that $x\notin B$. We denote by $[a]$ the equivalence class of an element $a$ of a set $A$ under an equivalence relation on $A$. We denote by $\binom{n}{r}$ the number of $r$-element subsets of an $n$-element set. We will use $A = \{n_1\cdot a_1, \ldots, n_k\cdot a_k\}$ to denote the multiset $A$ with $n_i$ copies of $a_i$ for each $i\in I_k$.

\vspace{0.2cm}

Let $X$ be a nonempty set. A \emph{partition} of $X$ is a collection of nonempty disjoint subsets of $X$, called \emph{blocks}, whose union is $X$. A partition is called \emph{trivial} if it has only singleton blocks or a single block. A partition is called \emph{uniform} if all its blocks have the same size. An \emph{$m$-partition} is a partition that has exactly $m$ blocks. For $m, k \in \mathbb{N}$ with $m \ge k$, an \emph{$(m, k)$-partition} is an $m$-partition that has exactly $k$ different size blocks. It is well-known that any partition of $X$ induces naturally an equivalence relation on $X$, and vice versa (cf. \cite[Proposition 1.4.6]{howie95}).

\vspace{0.2cm}

The composition of maps will be denoted by juxtaposition. A \emph{selfmap} on a set $A$ is a map from $A$ to $A$. Let $f, g \in \mathcal{T}_X$. For $x\in X$, we will use $xf$ to denote the image of $x$ under $f$, and compose maps from left to right: $x(fg) = (xf)g$. The symbols $\mbox{dom}(f)$ and $\mbox{codom}(f)$ will be used to denote respectively the domain and the codomain of $f$. The pre-image of a subset $B\subseteq X$ under $f$ is denoted by $Bf^{-1} = \{x\in X\;|\; xf\in B\}$. If $A, B \subseteq X$ such that $Af \subseteq B$, then there is a map $g \colon A \to B$ defined by $xg = xf$ for all $x\in A$ and, in this case, we say that $g$ \emph{is induced by} $f$. If $X =\{1,\ldots, n\}$, we will write $\mathcal{T}_X$ and $\mathcal{S}_X$ as $T_n$ and $S_n$, respectively.

\vspace{0.2cm}

A permutation $f$ of $I_n$ is called a \emph{cycle of length $m$} or an \emph{$m$-cycle}, denoted by $(a_1, \ldots, a_m)$, if there exists a subset $\{a_1, \ldots, a_m\}$ of $I_n$ such that $a_i f = a_{i+1}$ for all $i = 1,\ldots, m-1$, $a_m f  = a_1$, and $af = a$ for all $a\notin\{a_1, \ldots, a_m\}$. Any two cycles $(a_1, \ldots, a_k)$ and $(b_1, \ldots, b_l)$ are said to be \emph{disjoint} if the subsets $\{a_1, \ldots, a_k\}$ and $\{b_1, \ldots, b_l\}$ are disjoint. It is well-known that every permutation is expressible as a composition of disjoint cycles (cf. \cite[Theorem 5.1]{gallian17}).

\vspace{0.2cm}

Let $S$ be a semigroup. An element $a\in S$ is called an \emph{idempotent} provided that $a^2 = a$. The set of idempotents of $S$ is denoted by $E(S)$. It is well-known that $f\in \mathcal{T}_X$ is an idempotent if and only if $f$ acts as the identity map on its image set (cf. \cite[pp. 6]{clifford61}). The \emph{group of units} of $S$ is the subgroup of all invertible elements of $S$. An equivalence relation $\rho$ on $S$ is called a \emph{congruence} if for all $x, y, z, w \in S),\; (x, y)\in \rho$ and  $(z, w) \in \rho$ implies $(xz, yw)\in \rho$. If $\rho$ is a congruence on $S$, then the factor set $S/\rho$ is a semigroup, called a \emph{quotient semigroup}, equipped with the multiplication defined by $[x] [y] = [xy]$ for all $[x], [y] \in S/\rho$. We will write $S \cong T$ to mean that there is an isomorphism between two semigroups $S$ and $T$.

\vspace{0.2cm}
\section{The Semigroup $\Sigma(X, \mathcal{P})$}

In this section, we first characterize the elements of the semigroup $\Sigma(X, \mathcal{P})$. We then prove two simple but important lemmas on $S(X, \mathcal{P})$ and $\Sigma(X,\mathcal{P})$, respectively.  We begin by recalling a definition.

\begin{definition}(cf. \cite{puri16})
Let $\mathcal{P}=\{X_i\;|\;i\in I\}$ be a partition of an arbitrary set $X$, and let $f \in T(X,\mathcal{P})$. The \emph{character} of $f$ is a selfmap $\chi^{(f)}\colon I \to I$ defined by $i\chi^{(f)}=j$ whenever $X_i f\subseteq X_j$.
\end{definition}

When $X$ is a finite set, the selfmap  $\chi^{(f)}$ has been discussed, and also denoted by $\overline{f}$ in \cite{arjo15, east-c16, east16}.

\vspace{0.2cm}
Denote by $S_{\mathcal{P}}(X)$ the semigroup of all continuous selfmaps on $X$ endowed with the topology having $\mathcal{P}$ as a basis. By \cite[Theorem 2.8]{pei-s94}, we know that $S_{\mathcal{P}}(X) = T(X, \mathcal{P})$ . We now have the following.

\begin{theorem}\label{chi-1}
Let $\mathcal{P}=\{X_i\;|\; i\in I\}$ be a partition of an arbitrary set $X$, and let $f\in T(X, \mathcal{P})$. Then the following statements are equivalent:
\begin{enumerate}
\item[\rm(i)] $f\in \Sigma(X,\mathcal{P})$.
\item[\rm(ii)] $\chi^{(f)}$ is a surjective map.
\item[\rm(iii)] $f\in S_{\mathcal{P}}(X)$ such that $Af^{-1}\neq \emptyset$ for all nonempty open set $A$.
\end{enumerate}
\end{theorem}
\begin{proof}
$\mbox{(i)}\Longrightarrow \mbox{(ii)}$. Let $j\in I$. Since $f\in \Sigma(X,\mathcal{P})$, we have $Xf\cap X_j\neq \emptyset$. Then there exists $i\in I$ such that $X_if\subseteq X_j$. It follows that $i\chi ^{(f)}=j$ by definition of $\chi^{(f)}$. Hence $\chi ^{(f)}$ is surjective.

\vspace{0.2cm}
$\mbox{(ii)}\Longrightarrow \mbox{(iii)}$. Let $A$ be a nonempty open set. Then $A=\bigcup_{j\in J}X_j$ for some $J\subseteq I$ (cf. \cite[Definition 2.2.1]{conway14}). Since $\chi^{(f)}\colon I \to I$ is surjective, for each $j\in J$ there exists $i\in I$ such that $i\chi^{(f)}=j$ and subsequently $X_i f \subseteq X_j$. That means $ X_jf^{-1}\neq \emptyset$ for all $j\in J$. Hence we obtain
\[Af^{-1} = \Big(\bigcup_{j\in J}X_j\Big)f^{-1}=\bigcup_{j\in J} (X_jf^{-1}) \neq \emptyset.\]

\vspace{0.2cm}
$\mbox{(iii)}\Longrightarrow \mbox{(i)}$.
Let $X_i \in \mathcal{P}$. By (iii), we have $X_if^{-1}\neq \emptyset$. It follows that $Xf\cap X_i\neq \emptyset$ and hence $f\in \Sigma(X,\mathcal{P})$.
\end{proof}

\begin{definition}[cf. \cite{pei-c05}]
Let $E$ be an equivalence relation on a set $X$. A selfmap  $f\colon X \to X$ is said to be \emph{$E^*$-preserving} if $f$ satisfies the following.
\[(x,y)\in E \mbox{ if and only if } (xf,yf)\in E.\]
\end{definition}

\vspace{0.1cm}
We next have the following.

\begin{theorem}\label{chi-2}
Let $\mathcal{P} = \{X_i\;|\; i\in I\}$ be the partition associated with an equivalence relation $E$ on an arbitrary set $X$, and let $f\in T(X, \mathcal{P})$. Then the following statements are equivalent:
\begin{enumerate}
\item[\rm(i)] $\chi^{(f)}$ is an injective map.
\item[\rm(ii)] $f$ is an $E^*$-preserving map.
\end{enumerate}
\end{theorem}
\begin{proof}
$\mbox{(i)}\Longrightarrow \mbox{(ii)}$. Let $x,y\in X$. Since $f\in T(X,\mathcal{P})$, it is clear that if $(x,y)\in E$ then $(xf,yf)\in E$. On the other hand, assume that $(xf,yf)\in E$. Then there exists $X_l\in \mathcal{P}$ such that $xf,yf\in X_l$. We now claim that $x, y \in X_r$ for some $X_r \in \mathcal{P}$.

\vspace{0.2cm}
On the contrary, suppose that $x\in X_s$ and $y\in X_t$ for two distinct blocks $X_s, X_t \in \mathcal{P}$. Since $xf,yf\in X_l$, and $f\in T(X, \mathcal{P})$, we have $X_sf\subseteq X_l$ and $X_tf\subseteq X_l$. Then $s\chi^{(f)}=l=t\chi^{(f)}$ by definition of $\chi^{(f)}$. This contradicts the hypothesis that $\chi^{(f)}$ is injective. Hence $x, y \in X_r$ for some $X_r \in \mathcal{P}$ and consequently $f$ is $E^*$-preserving.

\vspace{0.2cm}
$\mbox{(ii)}\Longrightarrow \mbox{(i)}$. By contradiction, suppose that there are two distinct elements $s, t \in I$ such that $s\chi ^{(f)}=t\chi ^{(f)}$, say equal to $r$. Then $X_sf\subseteq X_r$ and $X_tf\subseteq X_r$ by definition of $\chi^{(f)}$. Let $x\in X_s$ and $y\in X_t$. Since $X_s\cap X_t = \emptyset$, it is clear that $(x, y)\notin E$.

\vspace{0.2cm}
Recall that $X_sf\subseteq X_r$ and $x\in X_s$, it follows that $xf \in X_r$. Also $X_tf\subseteq X_r$ and $y\in X_t$, it follows that $yf \in X_r$. Thus $(xf, yf)\in E$ which contradicts the hypothesis that $f$ is an $E^{*}$-preserving map. Hence $\chi^{(f)}$ is an injective map.
\end{proof}

\vspace{0.2cm}
Note that any selfmap on a finite set that is injective or surjective must be bijective (cf. \cite[Proposition 1.1.3]{gan-maz09}). Combining this fact with Theorem \ref{chi-1} and Theorem \ref{chi-2}, we have the following direct corollary.

\begin{corollary}\label{4-equi-sigma}
Let $\mathcal{P}= \{X_1, \ldots, X_m\}$ be an $m$-partition associated with an equivalence relation $E$ on an arbitrary set $X$, and let $f\in T(X, \mathcal{P})$. Then the following four statements are equivalent:
\begin{enumerate}
\item[\rm(i)] $f\in \Sigma (X,\mathcal{P})$.
\item[\rm(ii)] $\chi^{(f)}$ is a bijective map on $I_m$.
\item[\rm(iii)] $f$ is an $E^*$-preserving map.
\item[\rm(iv)] $f\in S_{\mathcal{P}}(X)$ such that $Af^{-1}\neq \emptyset$ for all nonempty open set $A$.
\end{enumerate}
\end{corollary}

\vspace{0.1cm}
\begin{lemma}\label{image-set-block}
Let $\mathcal{P} = \{X_i\;|\; i\in I\}$ be a partition of an arbitrary set $X$, and let $f \in T(X, \mathcal{P})$. If $f \in S(X, \mathcal{P})$, then
\begin{enumerate}
\item[\rm(i)] $X_i f \in \mathcal{P}$ for all $X_i \in \mathcal{P}$.
\item[\rm(ii)] $|X_i| = |X_j|$ if $i\chi^{(f)} = j$.
\end{enumerate}
\end{lemma}

\begin{proof}
Note that any two blocks of a partition are either equal or disjoint.
\begin{enumerate}
\item[\rm(i)] On the contrary, suppose that $X_if \notin \mathcal{P}$. Then there exists $X_j \in \mathcal{P}$ such that $X_if \subsetneq X_j$.
Since $f \in S(X, \mathcal{P})$ and $S(X, \mathcal{P})$ is a subgroup, it follows that the inverse map $f^{-1}$ of the bijection $f$ belongs to $S(X, \mathcal{P})$. Then $X_j f^{-1} \subseteq X_k$ for some $X_k \in \mathcal{P}$. Thus we obtain \[X_i = (X_i f)f^{-1} \subsetneq X_j f^{-1} \subseteq X_k,\] which contradicts the fact that any two blocks are either equal or disjoint. Hence $X_if \in \mathcal{P}$.

\vspace{0.2cm}
\item[\rm(ii)] If $i\chi^{(f)} = j$, then $X_i f \subseteq X_j$ by definition of $\chi^{(f)}$.
By (i), we have $X_i f \in \mathcal{P}$. Recall that any two blocks are either equal or disjoint. Since $X_j \in \mathcal{P}$, we have $X_i f = X_j$ whence $|X_i| = |X_j|$ since $f$ is a bijection.
\end{enumerate}
\end{proof}

We next recall the following congruence on the semigroup $T(X,\mathcal{P})$ which is induced by $\chi^{(f)}$.
\begin{remark}(cf. \cite{puri16})
Let $\mathcal{P}=\{X_i\;|\;i\in I\}$ be a partition of an arbitrary set $X$. Then the relation $\chi$ on
$T(X,\mathcal{P})$ defined by \[(f, g)\in \chi \Longleftrightarrow \chi^{(f)} = \chi^{(g)}\] is a congruence on the semigroup $T(X,\mathcal{P})$.
\end{remark}

\vspace{0.2cm}
Note that the congruence $\chi$ defined above on the semigroup $T(X,\mathcal{P})$ is also a congruence on the semigroup $\Sigma(X,\mathcal{P})$, and the quotient semigroup $\Sigma(X,\mathcal{P})/\chi $ is a subsemigroup of $T(X,\mathcal{P})/\chi$. We now prove the following proposition.

\vspace{0.1cm}
\begin{lemma}\label{lemma-2}
Let $\mathcal{P}= \{X_i\;|\; i \in I\}$ be a partition of a set $X$. Then the semigroup
$\Sigma(X,\mathcal{P})/ \chi$ is isomorphic to the semigroup of all surjective selfmaps on $I$.
\end{lemma}

\begin{proof}
Denote by $\mathcal{O}(I)$ the semigroup of all surjective selfmaps on $I$. Define a map $\phi \colon \Sigma(X,\mathcal{P})/ \chi \to  \mathcal{O}(I)$ by setting $[f]\phi = \chi^{(f)}$ for all $[f]\in \Sigma(X,\mathcal{P})/\chi$. By Theorem \ref{chi-1}, it is clear that $\phi$ is well-defined.

\vspace{0.2cm}
Let $[f], [g]\in \Sigma(X,\mathcal{P})/ \chi$. Note that $\chi^{(fg)}=\chi^{(f)} \chi^{(g)}$ (cf. \cite[Lemma 2.3]{puri16}. Therefore, we obtain \[([f][g])\phi= \chi^{(fg)}=\chi^{(f)} \chi^{(g)}= ([f]\phi) ([g]\phi)\] whence $\phi$ is a homomorphism. To see $\phi$ is injective, suppose that $[f]\phi =[g]\phi$. Then \[\chi^{(f)}=\chi^{(g)}\implies (f,g)\in \chi \implies [f]=[g]\] and
$\phi $ is injective.

\vspace{0.1cm}
To show that $\phi$ is surjective, let $\psi \in \mathcal{O}(I)$. For each $j\in I$, we arbitrarily choose an element $x_{j}\in X_{j}$.
Define a map $f\colon  X\to X$ by setting $xf=x_j$ whenever $x\in X_i$ and $i\psi = j$. Since $\psi$ is surjective, one can verify, in a routine manner, that $f\in \Sigma(X,\mathcal{P})$ and $\chi ^{(f)}=\psi$. Therefore, $[f]\phi =\chi ^{(f)}=\psi$ whence $\phi$ is surjective. This completes the proof.
\end{proof}

\vspace{0.1cm}
\section{Permutations that Preserve Nontrivial Partition}

When $\mathcal{P}$ is a trivial partition of $X$, it is clear that $T(X, \mathcal{P}) = \mathcal{T}_X$ and $S(X, \mathcal{P}) = \mathcal{S}_X$. Observe that $T(X, \mathcal{P}) \subsetneq \mathcal{T}_X$ and $S(X, \mathcal{P}) \subsetneq \mathcal{S}_X$ for all nontrivial partition $\mathcal{P}$ of $X$. We will now naturally be concerned of the converse problem: for a map $f\in \mathcal{T}_X$ ($f\in \mathcal{S}_X$), does there exist a nontrivial partition $\mathcal{P}$ of $X$ such that $f\in T(X, \mathcal{P})$ ($f\in S(X, \mathcal{P})$)? This section deals with this interesting problem when $X$ is finite. In the rest of this section, let $X = \{1,\ldots, n\}$.

\vspace{0.2cm}
Let $f\in T_n \setminus S_n$. If $f$ is a constant map, it is clear that $f \in T(X, \mathcal{P})$ for each nontrivial partition $\mathcal{P}$ of $X$. Otherwise if $f$ is a nonconstant map, we see that $f\in T(X,\mathcal{P})$ where $\mathcal{P}$ is the nontrivial partition of $X$ induced by the equivalence relation $\ker(f)$ on $X$ defined as $\ker(f) = \{(x, y)\;|\; xf  = yf\}$. Thus,

\begin{remark}
If $f\in T_n \setminus S_n$, then there exists a nontrivial partition $\mathcal{P}$ of $X$ such that $f\in T(X,\mathcal{P})$.
\end{remark}

\begin{proposition}
Let $f\in S_n$. If $f$ is not an $n$-cycle, then there exists a nontrivial partition $\mathcal{P}$ of $X$ such that $f\in S(X,\mathcal{P})$.
\end{proposition}

\begin{proof}
The result is obvious for the identity permutation since the identity permutation preserves each partition of $X$. Therefore, suppose that $f$ is a nonidentity permutation. If $f$ is not an $n$-cycle, we can write \[f=\beta_1\beta_2\cdots\beta_s,\] where $\beta_i$'s are cycles of length less than $n$ (cf. \cite[Theorem 5.1]{gallian17}). Let $X_1=\{x\in X\;|\; x\beta_1\neq x\}$. It is clear that $X_1f = X_1$ and $(X\setminus X_1)f = (X\setminus X_1)$. Moreover, one can verify, in a routine manner, that $\mathcal {P}=\{ X_1, X\setminus X_1\}$ is a nontrivial partition of $X$. Hence, $f\in S(X,\mathcal{P})$ where $\mathcal{P}$ is a nontrivial partition of $X$.
\end{proof}

\begin{theorem}\label{k-divide-n}
Let $f\in S_n$ be an $n$-cycle, and let $m$ be an integer such that $1< m < n$. Then $m$ divides $n$ if and only if there exists a nontrivial $m$-partition $\mathcal{P}$ of $X$ such that $f\in S(X, \mathcal{P})$.
\end{theorem}

\begin{proof}
We first assume that $n=km$ for some integer $k$ with $1<k<n$. Without loss of generality, let $f=(1,\dots ,n)$ be an $n$-cycle. Consider, for each $1 \le i \le m$, the subset
\[X_i=\{i, i+m, i+2m,\ldots, i+(k-1)m\}.\]
of $X$. One can verify, in a routine manner, that $X_i\cap X_j= \emptyset$ for $i\neq j$, and $X=\bigcup_{i\in I_m}X_i$. This implies $\mathcal{P} = \{X_1, \ldots, X_m\}$ is a nontrivial $m$-partition of $X$. Since $f = (1,2,\ldots, n)$, we next see that $X_i f = X_{i+1}$ for $1\le i \le m-1$ and $X_m f = X_1$. Hence $f\in S(X,\mathcal{P})$ where $\mathcal{P}$ is a nontrivial $m$-partition of $X$.

\vspace{0.2cm}
Conversely, assume that $\mathcal{P} = \{X_1, \ldots, X_m\}$ is a nontrivial $m$-partition of $X$ such that $f\in S(X,\mathcal{P})$. Since $f \in S(X, \mathcal{P})$, by Lemma \ref{image-set-block}(i), we have $X_if \in \mathcal{P}$ for all $i\in I_m$. We first claim that $X_if \neq X_i$ for all $i\in I_m$. If possible, let $X_sf=X_s$ for some $s\in I_m$. Then $(X\setminus X_s)f= X\setminus X_s$. It follows that that $f$ is not an $n$-cycle, which is a contradiction. Hence $X_if \neq X_i$ for all $i\in I_m$.

\vspace{0.2cm}
Let $i, j \in I_m$ such that $i \neq j$. If $X_i f = X_j$, then we claim that $X_j f \neq X_i$. If possible, let $X_j f = X_i$. Then $f$ is not an $n$-cycle, which is a contradiction. Thus, without loss of generality, assume that
$X_i f = X_{i+1}$ for each $1 \le i \le m-1$ and $X_m f = X_1$. Since $f\in S(X,\mathcal{P})$, by Lemma \ref{image-set-block}(ii), we have $|X_1| =\cdots =|X_m|$. Then letting $|X_i|= t$ for each $i\in I_m$, we obtain
\[|X_1|+ \cdots + |X_m| = n\Longrightarrow tm = n\] whence $m$ divides $n$. This completes the proof.
\end{proof}

\vspace{0.1cm}
The following corollary is an immediate consequence of Theorem \ref{k-divide-n}.

\begin{corollary}
Let $f\in S_n$ be an $n$-cycle. If $n$ is a prime, then $f\notin S(X, \mathcal{P})$ for any nontrivial partition $\mathcal{P}$ of $X$.
\end{corollary}

\begin{proposition}
Let $\mathcal{P} = \{X_1, \ldots, X_m\}$ be an $m$-partition of $X$, and let $f\in S(X,\mathcal{P})$. If $f$ is an $n$-cycle, then
\begin{enumerate}
\item[\rm(i)] $\chi^{(f)}$ is an $m$-cycle on $I_m$.
\item[\rm(ii)] $\mathcal{P}$ is a uniform partition of $X$.
\end{enumerate}
\end{proposition}

\begin{proof}\
\begin{enumerate}
\item[\rm(i)] Note that $ S(X,\mathcal{P})\subseteq \Sigma (X,\mathcal{P})$. In view of Corollary \ref{4-equi-sigma}, $\chi^{(f)}$ is a permutation of $I_m$. If $f$ is an $n$-cycle,
then $m$ divides $n$ by Theorem \ref{k-divide-n}. If $m$ is an improper divisor of $n$, one can easily verify, in a routine manner, that $\chi^{(f)}$ is an $m$-cycle. Otherwise, assume that $m$ is a proper divisor of $n$.

\vspace{0.2cm}
On the contrary, suppose that $\chi^{(f)}$ is not an $m$-cycle. Then $\chi^{(f)}$ can be written as a product of disjoint cycles (cf. \cite[Theorem 5.1]{gallian17}). Let $(i_1, \ldots, i_t)$ be a $t$-cycle, $1< t < m$, in the cycle decomposition of $\chi^{(f)}$. Then $i_r \chi^{(f)} = i_{r+1}$ for $r = 1,\ldots, t-1$ and $i_t \chi^{(f)} = i_1$.

\vspace{0.2cm}
Since $f\in S(X,\mathcal{P})$, by Lemma \ref{image-set-block}(i) we thus obtain $X_{i_r}f = X_{i_{r+1}}$ for all $r = 1,\ldots, t-1$ and $X_{i_t} f = X_{i_1}$. This means that $f$ is a cycle of length at most $tk$, which is obviously less than $n$, where $k= |X_{i_j}|$ for each $j = 1,\ldots, t$. This gives a contradiction of the hypothesis that $f$ is an $n$-cycle. Hence $\chi^{(f)}$ is an $m$-cycle on $I_m$.

\vspace{0.2cm}
\item[\rm(ii)] Without loss of generality, assume that $\chi^{(f)} = (1,\ldots, m)$ by (i). This means $i\chi^{(f)} = i+1$ for all $i = 1,\ldots, m-1$ and $m\chi^{(f)} = 1$. Since $f\in S(X,\mathcal{P})$, by Lemma \ref{image-set-block}(ii) we have $|X_1| = \cdots = |X_m|$ whence $\mathcal{P}$ is a uniform partition of $X$.
\end{enumerate}
\end{proof}

\section{Idempotents in $\Sigma(X, \mathcal{P})$}

In this section, we first define a notion of block maps associated with a set partition. We then proceed with lemmas and a corollary. By using these lemmas, we next characterize the idempotents of the semigroup $\Sigma(X, \mathcal{P})$. We also prove a necessary and sufficient condition for a map of $T(X, \mathcal{P})$ that belongs to $S(X, \mathcal{P})$. We finally count the idempotents in the semigroup $\Sigma(X, \mathcal{P})$ when $X$ is finite.

\vspace{0.1cm}
\begin{definition}
Let $\mathcal{P}$ be a partition of an arbitrary set. A \emph{block map} is a map whose both domain and codomain are blocks of $\mathcal{P}$.
\end{definition}

\vspace{0.1cm}
We now prove the following simple but important lemmas.
\begin{lemma}\label{fam-func}
Let $\mathcal{P} = \{X_i\;|\;i\in I\}$ be a partition of a set $X$, and let $f \in \mathcal{T}_X$. Then $f\in T(X,\mathcal{P})$ if and only if there exists a unique indexed family $B(f, I)$ of block maps induced by $f$, where \[B(f, I) = \{f_i \;|\;  f_i \mbox{ is induced by }f\mbox{and } \mbox{dom}(f_i) = X_i \mbox{ for each }i\in I\}.\]
\end{lemma}

\begin{proof}
We first assume that $f\in T(X,\mathcal{P})$, and let $X_i\in\mathcal{P}$. Then there exists $X_j \in \mathcal{P}$ such that $X_if\subseteq X_j$. We subsequently have a block map from $X_i$ to $X_j$ induced by $f$. Denote this induced block map by $f_i$. Since $X_i \in \mathcal{P}$ is an arbitrary block, we therefore obtain a unique indexed family $B(f, I)$ of block maps induced by $f$, where
$B(f, I) = \{f_i \;|\;  f_i \mbox{ is induced by }f \mbox{ and } \mbox{dom}(f_i) = X_i \mbox{ for each }i\in I\}$.

\vspace{0.2cm}
Conversely, suppose that the condition holds. Let $X_i \in \mathcal{P}$. Then $X_i f = X_i f_i$. Since $\mbox{codom}(f_i) \in \mathcal{P}$, we have $X_i f \subseteq X_j$, where $X_j = \mbox{Codom}(f_i)$. Hence $f\in T(X, \mathcal {P})$. This completes the proof.
\end{proof}

\vspace{0.1cm}
\begin{lemma}\label{f-idem}
Let $\mathcal{P} = \{X_i\;|\;i\in I\}$ be a partition of a set $X$, and let $f\in T(X,\mathcal{P})$. If each block map of $B(f, I)$ is an idempotent, then $f$ is an idempotent.
\end{lemma}
\begin{proof}
Let $x \in X$. Then $x\in X_i$ for some $i\in I$. Since $f_i \in B(f, I)$ is an idempotent, it follows that $f_i\colon X_i \to X_i$. Therefore
\[x(f^2) = (xf)f = (xf_i)f = (xf_i)f_i = x(f_i^2) = xf_i = xf\]
whence $f$ is an idempotent.
\end{proof}

\vspace{0.1cm}
\begin{lemma}\label{block-idem}
Let $\mathcal{P} = \{X_i\;|\;i\in I\}$ be a partition of a set $X$, and let $f\in E(T(X,\mathcal {P}))$. If $i \in I\chi^{(f)}$, then the block map $f_i \in B(f, I)$ is an idempotent.
\end{lemma}

\begin{proof}
If $i \in I\chi^{(f)}$, then we see that $Xf \cap X_i \neq \emptyset$. Let $y\in Xf \cap X_i$. Since $f\in E(T(X,\mathcal {P}))$, we have $yf = y$ (cf. \cite[pp. 6]{clifford61}). If follows that $f_i \colon X_i \to X_i$. Let $x\in X_i$. Then
\[x(f_i^2)= (xf_i)f_i =  (xf)f_i= x(f^2) = xf=xf_i\]
whence the block map $f_i$ is an idempotent.
\end{proof}

\vspace{0.1cm}
\begin{corollary}
Let $\mathcal{P} = \{X_i\;|\;i\in I\}$  be a partition of a set $X$, and let $f\in T(X,\mathcal {P})$. If $f$ is an idempotent, then $\chi^{(f)}$ is an idempotent.
\end{corollary}

\begin{proof}
Let $j\in I\chi^{(f)}$. It is sufficient to show that $j\chi^{(f)} = j$ (cf. \cite[pp. 6]{clifford61}).
Since $j\in I\chi^{(f)}$, the block map $f_j \in B(f, I)$ is an idempotent by  Lemma \ref{block-idem}. It follows that $f_j\colon X_j \to X_j$ and subsequently $j\chi^{(f)} = j$. Hence $\chi^{(f)}$ is an idempotent.
\end{proof}

\vspace{0.2cm}
The following theorem which characterizes the idempotents of the semigroup $\Sigma (X,\mathcal {P})$ is a direct consequence of Lemma \ref{f-idem} and Lemma \ref{block-idem}.

\begin{theorem}\label{idem-sigma}
Let $\mathcal{P} = \{X_i\;|\;i\in I\}$ be a partition of a set $X$, and let $f\in \Sigma(X,\mathcal {P})$. Then $f$ is an idempotent if and only if each block map of $B(f, I)$ is an idempotent.
\end{theorem}

\vspace{0.1cm}
Theorem \ref{idem-sigma} yields the following.

\begin{corollary}
Let $\mathcal{P} = \{X_i \;|\; i \in I\}$ be a partition of a set $X$, and let $f \in \Sigma(X, \mathcal{P})$. If $f$ is an idempotent, then $\chi^{(f)}$ is the identity map.
\end{corollary}

\begin{proof}
Let $i\in I$. If $f \in \Sigma(X, \mathcal{P})$ is an idempotent, then $f_i\colon X_i \to X_i$ by Theorem \ref{idem-sigma}.
It follows that $X_i f = X_i f_i \subseteq X_i$. By definition of $\chi^{(f)}$, we then have $i\chi^{(f)} = i$. Hence $\chi^{(f)}$ is the identity map on $I$.
\end{proof}


The next theorem provides a necessary and sufficient condition for a map of $T(X, \mathcal{P})$ that belongs to $S(X, \mathcal{P})$.
\begin{theorem}\label{char-sxp}
Let $\mathcal{P} = \{X_i \;|\; i \in I\}$ be a partition of an arbitrary set $X$, and let $f \in T(X, \mathcal{P})$. Then $f \in S(X, \mathcal{P})$ if and only if
\begin{enumerate}
\item[\rm(i)] each block map of $B(f, I)$ is a bijective map; and
\item[\rm(ii)] $\chi^{(f)}$ is a bijective map.
\end{enumerate}
\end{theorem}

\begin{proof}
Assume that $f \in S(X, \mathcal{P})$. Let $f_i \in B(f, I)$. It is clear that the block map $f_i$ is injective. If $\mbox{Codom}(f_i) = X_j$, then $X_i f = X_i f_i \subseteq X_j$. Since $f\in S(X,\mathcal{P})$, we have $X_i f_i \in \mathcal{P}$ by Lemma \ref{image-set-block}(i). We then have $X_i f_i  = X_j$ and so $f_i$ is surjective. Since $f_i \in B(f, I)$ is arbitrary, this proves (i).

\vspace{0.2cm}
We next prove (ii). Note that $S(X,\mathcal{P})\subseteq \Sigma(X,\mathcal{P})$. Since $f \in S(X,\mathcal{P})$, by Theorem \ref{chi-1}, the map $\chi^{(f)}$ is surjective. On the contrary, suppose that there exist two distinct elements $i, j \in I$ such that $i\chi ^{(f)}=j\chi ^{(f)}$, say equal to $k$ for some $k\in I$. Then $X_if \subseteq X_k$ and $X_jf \subseteq X_k$ by definition of $\chi ^{(f)}$. Since $f\in S(X, \mathcal{P})$, by Lemma \ref{image-set-block}(i), we have $X_if = X_k$ and $X_jf = X_k$ and subsequently $X_i = X_j$. This is a contradiction since $i\neq j$. Hence  the map $\chi^{(f)}$ is bijective.

\vspace{0.2cm}
Conversely, suppose that the conditions hold. It is sufficient to prove that $f$ is a bijection. We first prove that $f$ is injective. Let $x, y\in X$ and suppose that $xf=yf$. Then $xf, yf \in X_j$ for some $j\in I$. By (ii), there exists $i \in I$ such that $x, y \in X_i$. By (i), we know that $f_i$ is injective. We thus obtain $xf=yf\implies xf_i=yf_i\implies x=y$. Therefore $f$ is injective.

\vspace{0.2cm}
We now prove that $f$ is surjective. Let $y\in X$. Then $y\in X_j$ for some $j\in I$. By (ii), there exists $i\in I$ such that $i\chi^{(f)}=j$ and so $X_i f \subseteq X_j$. It follows that $f_i\colon X_i\to X_j$. By (i), there exists $x\in X_i$ such that $xf_i=y$ and subsequently $xf=y$. Therefore $f$ is surjective. This completes the proof.
\end{proof}

\vspace{0.2cm}

When $X$ is a finite set, the following theorem counts the number of the idempotents in the semigroup $\Sigma(X, \mathcal{P})$.

\begin{theorem}
Let $\mathcal{P}= \{X_1, \ldots, X_m\}$ be an $(m, k)$-partition of a finite set $X$ such that $\mathcal{P}$ has exactly $m_i\ge 1$ blocks of size $n_i\ge 1$ for each $i\in I_k$. Then
\[\big|E\big(\Sigma(X,\mathcal{P})\big)\big| = \prod_{i=1}^k{\bigg(\sum _{j=1}^{n_i}{n_i\choose j}j^{n_i-j}\bigg)^{m_i}}.\]
\end{theorem}

\begin{proof}
Note that each map $f \in \Sigma(X, \mathcal{P})$ is uniquely determined by the $m$-family $B(f, I_m)$ of block maps (cf. Lemma \ref{fam-func}).
From Theorem \ref{idem-sigma}, we know that a map $f\in \Sigma(X,\mathcal{P})$ is an idempotent if and only if each block map $f_i\in B(f, I_m)$ is an idempotent. It is therefore sufficient to count the total number of such $m$-families $B(f, I_m)$ of idempotent block maps. To count it, we break up the problem into $k$ subfamilies depending on the domain sizes of block maps.

\vspace{0.2cm}
Let $i\in I_k$. Since $\mathcal{P}$ has $m_i$ blocks of size $n_i$, we begin by counting the number of possible $m_i$-subfamilies of idempotent block maps from  $m_i$ distinct blocks of size $n_i$. Note that the total number of idempotents in the full transformation semigroup on an $n$-element set is $\sum _{j=1}^{n}{n\choose j}j^{n-j}$ (cf. \cite[Corollary 2.7.4]{gan-maz09}). Moreover, any map which is an idempotent must be selfmap. Therefore, the number of possible idempotent block maps from a block of size $n_i$ is $\sum_{j=1}^{n_i}{n_i\choose j}j^{n_i-j}$.

\vspace{0.2cm}
Recall that $\mathcal{P}$ has $m_i$ blocks of size $n_i$, by the multiplication principle, the total number of possible $m_i$-subfamilies of idempotent block maps from $m_i$ distinct blocks of size $n_i$ is $\Big(\sum_{j=1}^{n_i}{n_i\choose j}j^{n_i-j}\Big)^{m_i}$. Since $\mathcal{P}$ has exactly $k$ different size blocks and $i\in I_k$ is arbitrarily choosen element, the total number of possible $m$-families of idempotent block maps is now follows by applying the multiplication principle. This completes the proof.
\end{proof}


\section{The cardinality of  $T(X, \mathcal{P})$, $\Sigma(X, \mathcal{P})$, and $S(X, \mathcal{P})$}

This section calculates the size of $T(X, \mathcal{P})$, $S(X, \mathcal{P})$, and $\Sigma(X, \mathcal{P})$, respectively  when $X$ is a finite set. We begin by calculating the cardinality of the semigroup $T(X, \mathcal{P})$.

\begin{theorem}
Let $\mathcal{P}= \{X_1, \ldots, X_m\}$ be an $(m, k)$-partition of a finite set $X$ such that $\mathcal{P}$ has exactly $m_i\ge 1$ blocks of size $n_i\ge 1$ for each $i\in I_k$. Then
\[|T(X,\mathcal{P})|=\prod _{i=1}^k{\bigg(\sum _{j=1}^{k}m_jn_j^{n_i}\bigg)^{m_i}}.\]
\end{theorem}

\begin{proof}
Note that each map $f \in  T(X, \mathcal{P})$ is uniquely determined by the $m$-family $B(f, I_m)$ of block maps (cf. Lemma \ref{fam-func}). Therefore, it is sufficient to count the total number of such $m$-families $B(f, I_m)$ of block maps. To count it, we break up the problem into $k$ subfamilies depending on the domain sizes of block maps.

\vspace{0.2cm}
Let $i\in I_k$. Since $\mathcal{P}$ has $m_i$ blocks of size $n_i$, we begin by counting the number of possible $m_i$-subfamilies of block maps from $m_i$ distinct blocks of size $n_i$. Clearly, the codomain of a block map from a block of size $n_i$ can be any block of $\mathcal{P}$. Note that the number of maps from an $n$-element set into an $t$-element set is $t^{n}$. Therefore, the number of possible block maps from a block of size $n_i$ is $\sum_{j=1}^{k} m_j n_j^{n_i}$ by the addition principle.

\vspace{0.2cm}
Recall that $\mathcal{P}$ has $m_i$ blocks of size $n_i$, by the multiplication principle, the total number of possible $m_i$-subfamilies of block maps from $m_i$ distinct blocks of size $n_i$ is $\Big(\sum_{j=1}^{k} m_jn_j^{n_i}\Big)^{m_i}$. Since $\mathcal{P}$ has exactly $k$ different size blocks and $i\in I_k$ is arbitrarily choosen element, the total number of possible $m$-families of block maps is
now follows by applying the multiplication principle. This completes the proof.
\end{proof}

\vspace{0.2cm}
The next theorem calculate the cardinality of the group of units $S(X, \mathcal{P})$ of the semigroup $T(X, \mathcal{P})$.

\begin{theorem}
Let $\mathcal{P}= \{X_1, \ldots, X_m\}$ be an $(m, k)$-partition of a finite set $X$ such that $\mathcal{P}$ has exactly $m_i\ge 1$ blocks of size $n_i\ge 1$ for each $i\in I_k$. Then \[|S(X, \mathcal{P})| = \prod_{i = 1}^k (m_i!)(n_i!)^{m_i}.\]
\end{theorem}

\begin{proof}
Consider the equivalence relation $\sim$ on the partition $\mathcal{P}$ defined by $P \sim Q$ if and only if $|P| = |Q|$. Clearly, there are exactly  $k$ equivalence classes under the equivalence $\sim$. Let $[X_1], \ldots, [X_k]$ be the equivalence classes under $\sim$, where $|X_i| = n_i$.
Note that $|[X_i]| = m_i$ for each $i= 1,\ldots, k$.

\vspace{0.2cm}

Let $i\in I_k$. Consider the class $[X_i]$. By Lemma \ref{image-set-block} and Theorem \ref{char-sxp}(ii), we first note that the images of two distinct blocks in $[X_i]$ under a map of $S(X, \mathcal{P})$ must be distinct blocks in $[X_i]$. If $f\in S(X, \mathcal{P})$, then there are $m_i$ choices for the image of the first block of $[X_i]$ under $f$, the remaining $(m_i-1)$ choices for the image of the second block of $[X_i]$ under $f$, etc. For the last block of $[X_i]$, there is exactly one choice under $f$. Therefore, by the multiplication principle, all $m_i$ blocks of the class $[X_i]$ can be mapped in $m_i !$ different ways.

\vspace{0.2cm}

Note that the number of bijections between any two $n$-element sets is $n !$. Therefore, among $m_i !$ different choices, each fixed choice gives $(n_i !)^{m_i}$ bijections that preserve all blocks of $[X_i]$ by the multiplication principle. Hence, the total number of bijections that preserve all blocks of the class $[X_i]$ is $(m_i!)(n_i !)^{m_i}$.

\vspace{0.2cm}
Since there are exactly $k$ equivalence classes and $[X_i]$ is arbitrarily choosen class, one can obtain the stated formula of $|S(X, \mathcal{P})|$ by the multiplication principle.
\end{proof}

\vspace{0.2cm}

The following theorem calculate the cardinality of the semigroup $\Sigma(X,\mathcal{P})$.

\begin{theorem}
Let $\mathcal{P}= \{X_1, \ldots, X_m\}$ be an $(m, k)$-partition of a finite set $X$ such that $\mathcal{P}$ has exactly $m_i\ge 1$ blocks of size $n_i\ge 1$ for each $i\in I_k$. Then
\begin{equation*}
\begin{split}
|\Sigma(X,\mathcal{P})|&= m_1!\ldots m_k!\sum n_1^{s(m_1)}\ldots n_k^{s(m_k)},
\end{split}
\end{equation*}
where the sum runs over all $k$-tuple of
\[A := \big\{(t_{m_1}, \ldots, t_{m_k})\;|\;\forall i\in I_k,\; t_{m_i} = (l_1, \ldots, l_{m_i}), l_{j}\in \{n_1,  \ldots, n_k\}\big\}\]
such that the components of all $t_{m_i}$'s of a $k$-tuple of $A$ form
$\{m_1\cdot n_1, \ldots, m_k\cdot n_k\}$; each $s(m_i)$ is the sum of all components of $t_{m_i}$ in a $k$-tuple of $A$.
\end{theorem}

\begin{proof}
Since $\mathcal{P}$ is an $m$-partition of a finite set $X$, by Lemma \ref{lemma-2}, we must have $\Sigma(X,\mathcal{P})/\chi \cong S_m$. It follows that there are $m!$ distinct equivalence classes into which $\Sigma(X, \mathcal{P})$ splits under the equivalence $\chi$. Therefore, it is sufficient to calculate the cardinality of all $m!$ distinct equivalence classes under the equivalence $\chi$.

\vspace{0.2cm}
Let $[f] \in \Sigma(X,\mathcal{P})/\chi$ be an arbitrary class. We now calculate the cardinality of the class $[f]$. By Lemma \ref{fam-func}, it is sufficient to count the total number of $m$-families of block maps induced by maps of $[f]$. Note, for any $g, h \in [f]$, that $\chi^{(g)} = \chi^{(h)}$. From Corollary \ref{4-equi-sigma}, we recall that the map $\chi^{(f)}$ is a bijection on $I_m$.

\vspace{0.2cm}
Let $X_i \in \mathcal{P}$. If $i\chi^{(f)} = i'$, then  $X_i f_i = X_i f \subseteq X_{i'}$. Therefore the number of block maps from $X_i$ to $X_{i'}$ is $r_{i'}^{r_i}$, where $r_i, r_{i'} \in \{m_1\cdot n_1, \ldots, m_k\cdot n_k\}$ and $|X_i|=r_i, |X_{i'}|=r_{i'}$. Since $\mathcal{P}$ has exactly $m$ distinct blocks, by the multiplication principle, the total number of possible $m$-families $B(f, I_m)$ of block maps is
$r_{1'}^{r_1}\;r_{2'}^{r_2}\ldots r_{m'}^{r_m}$. Thus $|[f]| = r_{1'}^{r_1}\;r_{2'}^{r_2}\ldots r_{m'}^{r_m}$, where $r_i, r_{i'} \in \{m_1\cdot n_1, \ldots, m_k\cdot n_k\}$ for all $1\le i, i'\le m$.

\vspace{0.2cm}
Since $[f] \in \Sigma(X,\mathcal{P})/\chi$ is an arbitrary class, by the addition principle, we thus obtain
\begin{equation*}
\begin{split}
|\Sigma(X,\mathcal{P})|&=\sum_{[f]\in \Sigma(X,\mathcal{P})/ \chi} |[f]|
=\sum _{\phi \in S_m}r_{1'}^{r_1}r_{2'}^{r_2}\ldots r_{m'}^{r_m},
\end{split}
\end{equation*}
where $\phi \in S_m$ denotes the isomorphic image of the class $[f]$ and $i\phi = i'$. Since $\mathcal{P}$ has exactly $m_i$ blocks of size $n_i$ for each $i\in I_k$, we see that all the $r_i$'s form the multiset $\{m_1\cdot n_1, \ldots, m_k\cdot n_k\}$, and also all $r_{i'}$'s form the multiset $\{m_1\cdot n_1, \ldots, m_k\cdot n_k\}$.

\vspace{0.2cm}
Hence, by \cite[Theorem 2.4.2]{bru10}, we obtain
\begin{equation*}
\begin{split}
|\Sigma(X,\mathcal{P})|&= m_1!\ldots m_k!\sum n_1^{s(m_1)}\ldots n_k^{s(m_k)},
\end{split}
\end{equation*}
where the sum runs over all $k$-tuple of
\[A := \big\{(t_{m_1}, \ldots, t_{m_k})\;|\;\forall i\in I_k,\; t_{m_i} = (l_1, \ldots, l_{m_i}), l_{j}\in \{n_1,  \ldots, n_k\}\big\}\]
such that the components of all $t_{m_i}$'s of a $k$-tuple of $A$ form
$\{m_1\cdot n_1, \ldots, m_k\cdot n_k\}$; each $s(m_i)$ is the sum of all components of $t_{m_i}$ in a $k$-tuple of $A$.

\end{proof}

\section*{Acknowledgment}

The authors would like to thank the anonymous referee who provided critical and detailed comments on part of an earlier version of the manuscript.

\end{document}